\documentclass[reqno]{amsart}

\usepackage{style}

\author{Mireille Boutin}
\address{Department of Mathematics, Purdue
  University, 150 N.~University St., West Lafayette, IN, USA 47907}
\email{mboutin@purdue.edu}
\author{Gregor Kemper} \address{Technische Universit\"at M\"unchen,
  Zentrum Mathematik - M11, Boltzmannstr. 3, 85748 Garching, Germany}
\email{kemper@tum.de}

\title{Multilateration and Signal Matching with Unknown Emission Times}

\date{June, 2022}


\begin{document}

\begin{abstract}
  Assume that a source emits a signal in $3$-dimensional space at an unknown time, which is received by at least~$5$ sensors. In almost all cases the emission time and source position can be worked out uniquely from the knowledge of the times when the sensors receive the signal. The task to do so is the multilateration problem. But when there are several emission events originating from several sources, the received signals must first be matched in order to find the emission times and source positions. In this paper, we propose to use algebraic relations between reception times to achieve this matching. A special case occurs when the signals are actually echoes from a single emission event. In this case, solving the signal matching problem allows one to reconstruct the positions of the reflecting walls. We show that, no matter where the walls are situated, our matching algorithm works correctly for almost all positions of the sensors. 
  
  In the first section of this paper we consider the multilateration problem, which is equivalent to the GPS-problem, and give a simple algebraic solution that applies in all dimensions. 
\end{abstract}

\maketitle

\section*{Introduction} \label{sIntro}%

Positioning is a ubiquitous  problem in engineering. For example, one may want to determine the location of an object such as a vehicle, locate an event such as an earthquake, calibrate an array of devices such as microphones, or draw the map of an environment such as a building. In many scenarios, the objects to be located can emit a signal. In such case, one can use an array of receivers with known geometry to determine the objects location with respect to the position of the receiving array.  Alternatively, the objects may be equipped with a receiver so to be located with an array of sources with known geometry. While both problems are dual to each other, their difficulty and conditioning can vary significantly depending on the specific setup scenario and constraints imposed. In some cases, the underlying mathematical problem may actually be ill-posed.

This paper is concerned with multilateration, which is the task of determining the position of one or more sources emitting a wave signal (e.g., electromagnetic, acoustic, or seismic waves).   More specifically, we are trying to determine the position of sources sending out a signal from measurements of the times when this signal is received by various sensors situated at known positions. We assume that the clocks on the receivers are synchronized together, but not with the clock of the sources. In other words, the time of signal emission is unknown to the sensors, and thus the differences of arrivals (TDOAs) are the only meaningful available information. Therefore one also speaks of pseudo-range multilateration.

The literature on the well-posedness of multilateration problems is sparse. As far as we know, even the well-known GPS positioning problem, which we analyze in Section \ref{sMultilateration},  has not been thoroughly studied.  This corresponds to the problem of determining the position of one emission event (at an unknown time) received by several (synchronized) sensors at known locations.  It turns out that, even in this simple scenario, the TDOAs may not uniquely determine the position of the source.

More generally, we consider the case of several sources emitting undistinguishable signals at unknown times (e.g., earthquakes~[\citenumber{sha2021reliability}] or gunshots).  Clearly the order in which the signals arrive at the sensors can be wildly different from the order in which the signals were emitted. Therefore before feeding the reception times into a multilateration algorithm, a {\em matching} must be performed, identifying those received signals that come from the same emission event with each other. Ideally, this matching process should also discard spurious signals that are registered by just one sensor, i.e., not match such signals to any others. Note that a mistake in the matching process, or accidentally including a spurious signal in a match, will result in determining an emission event (time and position of emission) which never took place.

The problem of matching the sound events produced at a known time was previously studied in 
[\citenumber{DPWLV1}, \citenumber{jager2016room}, \citenumber{el2017time},\citenumber{Boutin:Kemper:2019}]. 
In the case where the times of emission are unknown, the problem is known as  ``TDOA disambiguation." Two sources of TDOA ambiguity are considered in the literature. The first one is the  ``multipath ambiguity" which is caused by the reverberation of the signal on objects in the environment and leads to spurious events. The triangle inequality, a zero cyclic sum condition, and characteristics of the  cross-correlation and autocorrelation of the microphone signals are used to help disambiguate such cases in [\citenumber{kreissig2013fast},\citenumber{zannini2010improved}]. 
The second one is the ``multisource ambiguity" resulting from multiple sources emitting the same signal. A few of the false matches due to this can be ruled out using the triangle inequality.
But, as far as we know, a more rigorous criterion for the case where the times of emission are unknown has not previously been proposed.

In this paper we show that if there are (at least)~$5$ sensors in $3$-dimensional space, then the reception times (more precisely, TDOAs) of a signal coming from a single emission event satisfy a certain algebraic relation. We propose to use this relation to perform the signal matching. So if a selection of signal reception times, one for each sensor, satisfies this relation, then these reception times are accepted as coming from the same event (see \cref{2aDetect}). Moreover, spurious signals registered by just one sensor will almost certainly not satisfy the relation and will therefore not be included in a match.

Of course the "almost" in the previous sentence is an important issue. In fact, there is no way to rule out the possibility that a spurious signal is registered at such an unlucky time that our matching algorithm, or any other algorithm based on the available information, falsely includes this signal in a match. This is true not only for spurious signals but also for signals coming from a real emission event that happened at an unlucky time. In particular, there is no way to position the sensors such that this possibility can be ruled out.

The situation becomes different, however, if the emission events are in fact just echoes from a single event. More precisely, consider an arrangement of flat surfaces ("walls") that reflect a sound signal emitted from an omnidirectional loudspeaker. Assuming the signal is of high frequency, we use the ray acoustics approximation. This means that the signals are virtually emitted from the point given by reflecting the loudspeaker position at the walls, and all (virtual) emissions happen {\em simultaneously}. Now we bring in~$5$ microphones at known positions. These record the echoes of the sound emission and feed them into the matching algorithm. After that, the wall positions can be determined by multilateration. Notice that in contrast to our previous paper~[\citenumber{Boutin:Kemper:2019}], we do not assume that the loudspeaker and the microphones have synchronized clocks and communicate the times of signal emission time. So the common emission time is still unknown, and pseudo-range multilateration is required. In \cref{3tGood} in this paper, we show that in this situation almost all microphone positions are good, in the sense that no false matches can happen. As explained above, this is in contrast to the situation where the (virtual) emission events are not assumed to be simultaneous. The proof of the theorem uses methods from computational commutative algebra. This is something that it has in common with the proof of the main results from~[\citenumber{Boutin:Kemper:2019}]. However, when we designed the proof of \cref{3tGood} we were surprised to find that the difficulties that arose were quite different from those in~[\citenumber{Boutin:Kemper:2019}].

Even though the emphasis of this paper may lie on the matching problem and on multiple wall detection, we also study the pseudo-range multilateration problem itself. In fact, we present a simple algebraic solution algorithm. For simplicity, we formulate this in three dimensions, but it really works for all dimensions $> 1$. We give a self-contained and rigid proof for the validity. It is well-known that with just~$4$ sensors, the pseudo-range multilateration problem usually has~$2$ solutions. We show by examples that this may also happen if there are~$5$ sensors, even if no~$4$ of them are coplanar. This is not a shortcoming of our algorithm: in our example, the available information of TDOAs simply does not allow to disambiguate the solutions. \\

The paper is organized in three sections. The first section introduces the notation and discusses the pseudo-range multilateration problem, giving the solution algorithm in \cref{1tLateration}. In the second section we then turn our attention to the case of multiple emission events. We present and prove the relation that holds between reception times coming from the same event, and derive the matching algorithm (\cref{2aDetect}) from this. The final section deals with the situation of matching echoes from a single sound event. The main result (\cref{3tGood}) from that section says that, loosely speaking, almost all microphone positions are good.

\par{\bf Acknowledgments.} A great part of this work was done during a research stay of the authors at the Mathematical Sciences Research Institute (MSRI) in Berkeley within the 2022 Summer Research in Mathematics program. We thank the MSRI team for creating a uniquely stimulating atmosphere at the institute. The program provided us a with a perfect research environment and a chance to concentrate fully on getting things done.
This work has also benefited  from a research stay of the two authors at the Banff International Research Station for Mathematical Innovation and Discovery (BIRS) under the ``Research in Teams'' program. We would like to thank BIRS for its hospitality and for providing an optimal working environment. We also thank Stefan Wetkge and Timm Oertel for stimulating discussions.

\section{Multilateration and the GPS-problem}
\label{sMultilateration}

\newcommand{\ve}[1]{\mathbf{#1}}%
\newcommand{\vertex}[1][black]{\node[#1,circle, draw, outer sep=3pt,
  inner sep=0pt, minimum size=3pt,fill]}
\newcommand{\diag}{\operatorname{diag}}

In this section we look at the pseudo-range multilateration problem, which is
equivalent to the GPS-problem (see below), and present two simple algebraic (``direct", as opposed to iterative)
solutions, see \cref{1tLateration}. This topic has received
considerable interest in the literature (see \mycite{Bancroft:1985},
\mycite{Krause:1987}, \mycite{Chaffee:Abel:1994}, \mycite{li2010design}, \mycite{lundberg2001alternative}, and
\mycite{Beck:Pan:2012}). 
But our results are general and appear to be new.

Our primary interest is in the following situation: A source at an
unknown position~$\ve x \in \RR^3$ emits a signal, in our applications
usually by sound, at an unknown time~$t$. (In fact, everything we are
about to say can easily be adapted to $\RR^n$ with $n \ge 2$, but not
to $\RR^1$, see \cref{1rDim}.)  There are~$m$ sensors at known
positions $\ve a_1 \upto \ve a_m \in \RR^3$. They receive the signal
at times~$t_1 \upto t_m$. We choose the unit of time such that the
signal propagation speed becomes~$1$. So we have
\begin{equation} \label{1eqBasic}%
  \lVert\ve a_i - \ve x\rVert = t_i - t \qquad (i = 1 \upto m)
\end{equation}
The task now is to work out the position~$\ve x$ and the emission
time~$t$. The very same equations arise if there are~$m$ sources at
known positions~$\ve a_i$ emitting signals at known times, which are
then received by a device at an unknown position~$\ve x$. In this case
the~$t_i$ are the differences between the reception times {\em
  according to the clock on the device} and the emission times
according to the (near-perfect) clocks on the sources, and~$t$ is the
(unknown) bias between the clock on the receiver and the clocks on the
sources. This is the GPS-problem.

We will work with the slightly weaker equations
\begin{equation} \label{1eqAbsolute}%
  \lVert\ve a_i - \ve x\rVert = |t_i - t| \qquad (i = 1 \upto m).
\end{equation}
Writing
$L := \left(\begin{smallmatrix} -1 \\ & 1 \\ && 1 \\ &&&
    1 \end{smallmatrix}\right)$,
$\tilde{\ve a_i} := \left(\begin{smallmatrix} t_i \\ \ve
    a_i\end{smallmatrix}\right)$ and
$\tilde{\ve x} := \left(\begin{smallmatrix} t \\ \ve
    x\end{smallmatrix}\right)$, we have
\begin{multline*}
  \lVert\ve a_i - \ve x\rVert^2 - (t_i - t)^2 = (\tilde{\ve a_i} -
  \tilde{\ve x})^T \cdot L \cdot (\tilde{\ve a_i} - \tilde{\ve x}) =
  \tilde{\ve a_i}^T L \tilde{\ve a_i} - 2 \tilde{\ve a_i}^T L \tilde{\ve
    x} + \tilde{\ve x}^T L \tilde{\ve x} = \\
  \lVert\ve a_i\rVert^2 - t_i^2 + 2 t_i t - 2 \ve a_i^T \ve x +
  \lVert\ve x\rVert^2 - t^2,
\end{multline*}
so~\eqref{1eqAbsolute} is equivalent to
\begin{equation} \label{1eqEquiv}%
  - 2 t_i t + 2 \ve a_i^T \ve x - \lVert\ve x\rVert^2 + t^2 =
  \lVert\ve a_i\rVert^2 - t_i^2 \qquad (i = 1 \upto m).
\end{equation}
We form the matrix
\begin{equation} \label{1eqA}%
  A :=
  \begin{pmatrix}
    - 2 t_1 & 2 \ve a_1^T & -1 \\
    \vdots & \vdots & \vdots \\
    -2 t_m & 2 \ve a_m^T & -1
  \end{pmatrix} \in \RR^{m \times 5},
\end{equation}
which contains only known quantities. With this, \eqref{1eqEquiv} can
be expressed as a system of linear equations for the unknown
quantities:
\begin{equation} \label{1eqMatrix}%
  A \cdot
  \begin{pmatrix}
    t \\
    \ve x \\
    \lVert\ve x\rVert^2 - t^2
  \end{pmatrix} =
  \begin{pmatrix}
    \lVert\ve a_1\rVert^2 - t_1^2 \\
    \vdots \\
    \lVert\ve a_m\rVert^2 - t_m^2
  \end{pmatrix}.
\end{equation}

Then this has the same solutions as~\eqref{1eqAbsolute}. Now we make the
assumption that $A$ has rank~$5$, so~\eqref{1eqMatrix} has a unique
solution. (The existence of a solution follows from the fact that the
point~$\ve x$ and time~$t$ of emission exist.) In the case $m = 5$ we
can then simply invert~$A$. If $m > 5$, we could delete all but~$5$
linearly independent equations from~\eqref{1eqMatrix}, which would
give an algebraically equivalent system with invertible matrix. But in
the real world there are inaccurate measurements, so it should be
wiser to apply the Moore-Penrose inverse
$(A^T A)^{-1} A^T \in \RR^{5 \times m}$. Specifically, if
$B \in \RR^{4 \times m}$ is obtained by deleting the last row from
$(A^T A)^{-1} A^T$, then
\begin{equation} \label{1eqB}%
  \begin{pmatrix}
    t \\
    \ve x
  \end{pmatrix} = B \cdot \begin{pmatrix}
    \lVert\ve a_1\rVert^2 - t_1^2 \\
    \vdots \\
    \lVert\ve a_m\rVert^2 - t_m^2
  \end{pmatrix}.
\end{equation}
So we have obtained a unique solution for the emission time and
place. This is our first solution, which is available only if $A$ has
rank~$5$. In \cref{1pAInvertible} we will say something about how
likely this it.

But first we will consider the case that $A$ has rank~$< 5$, and
derive our second solution. What we do assume is that the~$\ve a_i$
are not coplanar. This makes sense, since if the~$\ve a_i$ all lay in
the same plane, then even with a known emission time~$t$ the
location~$\ve x$ of the source could not be distinguished from the
point obtained by reflecting~$\ve x$ at this plane. Our assumption
amounts to saying that the matrix
\begin{equation} \label{eqTildeAdef}%
  \tilde A :=
  \begin{pmatrix}
    2 \ve a_1^T & -1 \\
    \vdots & \vdots \\
    2 \ve a_m^T & -1
  \end{pmatrix} \in \RR^{m \times 4}
\end{equation}
has rank~$4$ (see Assumption~A in~[\citenumber{Beck:Pan:2012}]), so in
particular we need $m \ge 4$. The Moore-Penrose inverse is
$\tilde B := (\tilde A^T \tilde A)^{-1} \tilde A^T \in \RR^{4 \times
  m}$, so $\tilde B \tilde A = I_4$. Now~\eqref{1eqEquiv} can be
restated as
\begin{equation} \label{1eqTildeA}%
  \tilde A \cdot 
  \begin{pmatrix}
    \ve x \\
    \lVert\ve x\rVert^2 - t^2
  \end{pmatrix} =
  2 t \begin{pmatrix}
    t_1 \\
    \vdots \\
    t_m
  \end{pmatrix} + \begin{pmatrix}
    \lVert\ve a_1\rVert^2 - t_1^2 \\
    \vdots \\
    \lVert\ve a_m\rVert^2 - t_m^2
  \end{pmatrix},
\end{equation}
and multiplying by $\tilde B$ yields
\begin{equation} \label{1eqUV}%
  \begin{pmatrix}
    \ve x \\
    \lVert\ve x\rVert^2 - t^2
  \end{pmatrix} = t \cdot
  \begin{pmatrix}
    \ve u \\
    \alpha
  \end{pmatrix} +
  \begin{pmatrix}
    \ve v \\
    \beta
  \end{pmatrix}, \ \text{where} \
  \begin{pmatrix}
    \ve u \\
    \alpha
  \end{pmatrix} := 2 \tilde B
  \begin{pmatrix}
    t_1 \\
    \vdots \\
    t_m
  \end{pmatrix} \ \text{and} \
  \begin{pmatrix}
    \ve v \\
    \beta
  \end{pmatrix} := \tilde B
  \begin{pmatrix}
    \lVert\ve a_1\rVert^2 - t_1^2 \\
    \vdots \\
    \lVert\ve a_m\rVert^2 - t_m^2
  \end{pmatrix}.
\end{equation}
Extracting components, we obtain the equivalent equations
\begin{equation} \label{1eqRank4}%
  \ve x = t \ve u + \ve v \quad \text{and} \quad \bigl(\lVert\ve
  u\rVert^2-1\bigr) t^2 + \bigl(2 \ve u^T \ve v-\alpha\bigr) t +
  \lVert\ve v\rVert^2 - \beta = 0.
\end{equation}
Observe that~$\ve u$, $\ve v$, $\alpha$, and~$\beta$ are all derived
from known quantities, so \eqref{1eqRank4} can be resolved. In the
following theorem, part~\ref{1tLaterationA} summarizes our result in
the rank-$5$ case, \ref{1tLaterationB} tells us that~\eqref{1eqRank4}
is actually equivalent to~\eqref{1eqAbsolute}, and~\ref{1tLaterationC}
says that the quadratic equation in~\eqref{1eqRank4} never
degenerates.

\begin{theorem} \label{1tLateration}%
  In the above situation and with the notation introduced, we have:
  \begin{enumerate}[label=(\alph*)]
  \item \label{1tLaterationA} If the matrix $A$ from~\eqref{1eqA} has
    rank~$5$, then~\eqref{1eqAbsolute} has a unique solution
    $\left(\begin{smallmatrix} t \\ \ve x\end{smallmatrix}\right)$
    given by~\eqref{1eqB}.
  \item \label{1tLaterationB} Assume that $\rank(A) \le 4$ and that
    the~$\ve a_i$ are not coplanar, so in particular $m \ge 4$. Then
    the equations~\eqref{1eqAbsolute} are satisfied by the same
    $\left(\begin{smallmatrix} t \\ \ve x\end{smallmatrix}\right)$
    as~\eqref{1eqRank4}.
  \item \label{1tLaterationC} Moreover, the coefficients of~$t^2$
    and~$t$ in the quadratic equation in~\eqref{1eqRank4} are not both
    zero, so~\eqref{1eqRank4} has one or, more likely, two solutions.
  \end{enumerate}
\end{theorem}

\begin{proof}
  Part~\ref{1tLaterationA} has already been shown, so we turn our
  attention to~\ref{1tLaterationB}. By renumbering we may assume that
  $\ve a_1 \upto \ve a_4$ are not coplanar. Then $\tilde A_4$, the
  upper $4 \times 4$-part of $\tilde A$, is invertible, and $A_4$, the
  upper $4 \times 5$-part of $A$, has rank~$4$. So we have a matrix
  $C \in \RR^{m \times 4}$,  with upper $4 \times 4$-part the identity matrix, such
  that $A = C \cdot A_4$ and $\tilde A = C \cdot \tilde
  A_4$. Extracting the first column from the first equation gives
  \begin{equation} \label{1eqT}%
    \begin{pmatrix}
      t_1 \\
      \vdots \\
      t_m
    \end{pmatrix} = C \cdot
    \begin{pmatrix}
      t_1 \\
      \vdots \\
      t_4
    \end{pmatrix}.
  \end{equation}
  Since~\eqref{1eqMatrix} has a solution $\left(\begin{smallmatrix} t \\
      \ve x \\ \lVert\ve x\rVert^2 - t^2\end{smallmatrix}\right)$, we
  obtain
  \begin{equation} \label{1eqC2}%
    \begin{pmatrix}
      \lVert\ve a_1\rVert^2 - t_1^2 \\
      \vdots \\
      \lVert\ve a_m\rVert^2 - t_m^2
    \end{pmatrix} = A
    \begin{pmatrix}
      t \\
      \ve x \\
      \lVert\ve x\rVert^2 - t^2
    \end{pmatrix} = C A_4
    \begin{pmatrix}
      t \\
      \ve x \\
      \lVert\ve x\rVert^2 - t^2
    \end{pmatrix} = C \cdot
    \begin{pmatrix}
      \lVert\ve a_1\rVert^2 - t_1^2 \\
      \vdots \\
      \lVert\ve a_4\rVert^2 - t_4^2
    \end{pmatrix}.
  \end{equation}
  We have
  \[
    \tilde B = (\tilde A^T A)^{-1} \tilde A^T = \bigl(\tilde A_4^T
    C^T C \tilde A_4\bigr)^{-1} \tilde A_4^T C^T = \tilde A_4^{-1}
    (C^T C)^{-1} C^T,
  \]
  so
  \begin{equation} \label{1eqMP}%
    \tilde A_4 \tilde B C = I_4 \quad \text{and} \quad \tilde A \tilde
    B C = C \tilde A_4 \tilde B C = C.
  \end{equation}

  To prove~\ref{1tLaterationB}, let
  $\left(\begin{smallmatrix} t \\ \ve x\end{smallmatrix}\right)$
  satisfy~\eqref{1eqRank4}. Then it satisfies~\eqref{1eqUV}, so
  \[
    \begin{pmatrix}
      \ve x \\ \lVert\ve x\rVert^2 - t^2
    \end{pmatrix} = 2 t \tilde B
    \begin{pmatrix}
      t_1 \\
      \vdots \\
      t_m
    \end{pmatrix} + \tilde B
    \begin{pmatrix}
      \lVert\ve a_1\rVert^2 - t_1^2 \\
      \vdots \\
      \lVert\ve a_m\rVert^2 - t_m^2
    \end{pmatrix} =
    2 t \tilde B C
    \begin{pmatrix}
      t_1 \\
      \vdots \\
      t_4
    \end{pmatrix} + \tilde B C
    \begin{pmatrix}
      \lVert\ve a_1\rVert^2 - t_1^2 \\
      \vdots \\
      \lVert\ve a_4\rVert^2 - t_4^2
    \end{pmatrix},
  \]
  where we used~\eqref{1eqT} and~\eqref{1eqC2}. With~\eqref{1eqMP}
  this implies
  \[
    \tilde A
    \begin{pmatrix}
      \ve x \\ \lVert\ve x\rVert^2 - t^2
    \end{pmatrix} = 2 t C
    \begin{pmatrix}
      t_1 \\
      \vdots \\
      t_4
    \end{pmatrix} + C
    \begin{pmatrix}
      \lVert\ve a_1\rVert^2 - t_1^2 \\
      \vdots \\
      \lVert\ve a_4\rVert^2 - t_4^2
    \end{pmatrix} = 2 t
    \begin{pmatrix}
      t_1 \\
      \vdots \\
      t_m
    \end{pmatrix} + \begin{pmatrix}
      \lVert\ve a_1\rVert^2 - t_1^2 \\
      \vdots \\
      \lVert\ve a_m\rVert^2 - t_m^2
    \end{pmatrix}.
  \]
  Thus~\eqref{1eqTildeA} holds, which is equivalent
  to~\eqref{1eqAbsolute}. So every solution of~\eqref{1eqRank4}
  satisfies~\eqref{1eqAbsolute}, and the converse is true
  since~\eqref{1eqRank4} was derived from~\eqref{1eqAbsolute}.

  To prove part~\ref{1tLaterationC}, assume that the coefficients
  of~$t^2$ and~$t$ in the quadratic equation in~\eqref{1eqRank4} are
  both zero. Since~\eqref{1eqAbsolute} has a solution, this means that
  also the constant coefficient vanishes, so there is a solution
  of~\eqref{1eqRank4}, and therefore of~\eqref{1eqAbsolute}, for
  every~$t$. In particular, we have $\ve x_0 \in \RR^3$ such that
  $\left(\begin{smallmatrix} 0 \\ \ve x_0\end{smallmatrix}\right)$
  satisfies~\eqref{1eqAbsolute}. Substituting each~$\ve a_i$
  by~$\ve a_i - \ve x_0$ preserves the noncoplanarity of
  $\ve a_1 \upto \ve a_4$ and yields new systems~\eqref{1eqAbsolute}
  and~\eqref{1eqRank4} of equations. For each solution
  $\left(\begin{smallmatrix} t \\ \ve x\end{smallmatrix}\right)$ of
  the old system~\eqref{1eqAbsolute}, the new one now has the solution
  $\left(\begin{smallmatrix} t \\ \ve x - \ve
      x_0\end{smallmatrix}\right)$, so the new quadratic equation
  in~\eqref{1eqRank4} still has infinitely many solutions. Moreover,
  $\left(\begin{smallmatrix} 0 \\ \ve 0\end{smallmatrix}\right)$ is a
  solution of~\eqref{1eqAbsolute}, so $\lVert\ve a_i\rVert =
  |t_i|$. Hence~\eqref{1eqUV} implies $\ve v = \ve 0$ and $\beta =
  0$. The vanishing of the coefficient of~$t$ in the quadratic equation therefore means $\alpha = 0$,
  therefore means $\alpha = 0$, so
  \begin{multline*}
    2 \begin{pmatrix}
      \ve a_1^T \ve u \\
      \vdots \\
      \ve a_4^T \ve u
    \end{pmatrix} =
    \begin{pmatrix}
      2 \ve a_1^T & -1 \\
      \vdots & \vdots \\
      2 \ve a_4^T & -1
    \end{pmatrix} \cdot
    \begin{pmatrix}
      \ve u \\
      \alpha
    \end{pmatrix} = \tilde A_4
    \begin{pmatrix}
      \ve u \\
      \alpha
    \end{pmatrix} \underset{\eqref{1eqUV}}{=} \\
    2 \tilde A_4 \tilde B
    \begin{pmatrix}
      t_1 \\
      \vdots \\
      t_m
    \end{pmatrix} \underset{\eqref{1eqT}}{=} 2 \tilde A_4 \tilde B C
    \begin{pmatrix}
      t_1 \\
      \vdots \\
      t_4
    \end{pmatrix} \underset{\eqref{1eqMP}}{=} 2
    \begin{pmatrix}
      t_1 \\
      \vdots \\
      t_4
    \end{pmatrix} = 2
    \begin{pmatrix}
      \varepsilon_1 \lVert\ve a_1\rVert \\
      \vdots \\
      \varepsilon_4 \lVert\ve a_4\rVert
    \end{pmatrix}
  \end{multline*}
  for some $\varepsilon_i \in \{\pm 1\}$. The vanishing of the $t^2$-coefficient
  in~\eqref{1eqRank4} means that $\lVert u \rVert = 1$, so
  \[
    |\ve a_i^T \ve u| = \lVert\ve a_i\rVert = \lVert\ve a_i\rVert
    \cdot \lVert u \rVert \quad (i = 1 \upto 4).
  \]
  Thus the Cauchy–Schwarz inequality is actually an equality, implying
  that~$\ve a_i$ and~$\ve u$ are linearly dependent. This shows that
  $\ve a_1 \upto \ve a_4$ are collinear, contradicting the hypothesis
  that they are not coplanar.
\end{proof}

\begin{rem} \label{1rDim}%
  Everything in this section carries over directly to $n$-dimensional
  space, with $n > 1$. Just replace every instance
  of~$3$, $4$, and~$5$ by~$n$, $n+1$, and~$n+2$, and replace
  ``coplanar'' by ``contained in a common $(n-1)$-dimensional affine
  subspace.''

  What happens for $n = 1$? Everything works, except for the very last
  sentence in the proof of \cref{1tLateration}: In dimension~$1$,
  collinearity does {\em not} contradict being noncopunctual (i.e.,
  not being the same point). But this makes everything break down. In
  fact, the mathematics bears out what has always been clear about the
  one-dimensional case: If the source lies on the same side of every
  sensor, then there is no way to find out its position from the time
  differences of signal arrivals; and indeed in this case the
  coefficients of the quadratic equation in~\eqref{1eqRank4} are
  all~$0$, and the matrix $A \in \RR^{m \times 3}$ has rank~$2$.

  Even though the one-dimensional case is not interesting in itself,
  it shows, as do~\ref{1exLateration5} and~\ref{1exLateration6} in
  \cref{1exLateration} below, that the effort of proving
  \cref{1tLateration}\ref{1tLaterationC} was not irrelevant: this is
  not a truism.
\end{rem}

Up to now, we have worked with the equations~\eqref{1eqAbsolute}, and
seen that they may have two solutions. But according
to~\eqref{1eqBasic} (which expresses that signals arrive after having
been sent) we have $t_i \ge t$. If one of the solutions does not
satisfy this, it is spurious and can be discarded. But if both do, the
given data do not uniquely determine~$\ve x$ and~$t$. The first two of
the following examples show that this can actually happen. The third
example has a spurious solution, and the others exemplify some special
cases.

\begin{ex} \label{1exLateration}%
  In the following, we chose the coordinates in a way to keep all
  numbers rational, so the examples are quite
  Pythagorean-triple-prone.%
  \newcommand{\comment}[1]{}
  \begin{enumerate}[label=(\arabic*)]
  \item \label{1exLateration1} Of the five sensor positions
    \[
      \ve a_1 = \begin{pmatrix} 3 \\ 4 \\ 0 \end{pmatrix}, \
      \ve a_2 = \begin{pmatrix} -2 \\ -2 \\ 1 \end{pmatrix}, \
      \ve a_3 = \begin{pmatrix} -1 \\ 0 \\ 0 \end{pmatrix}, \
      \ve a_4 = \frac{2}{21} \begin{pmatrix} 0 \\ -24 \\ 7 \end{pmatrix}, \
      \ve a_5 = \frac{1}{21} \begin{pmatrix} 0 \\ 76 \\ 0 \end{pmatrix},
    \]
    no four are coplanar.\ %
    \comment{
      Magma-code:

      n:=3;
      ai:=[[3,4,0],[-2,-2,1],[-1,0,0],[0,-16/7,2/3],[0,76/21,0]];
      Atilde:=Matrix(n+2,n+1,[[2*ai[i][j]: j in [1..n]] cat [-1]: i in
      [1..n+2]]);
      [Determinant(Submatrix(Atilde,Exclude([1..n+2],i),[1..n+1])): i
      in [1..n+2]];
      // All nonzero
    }%
    If we assume that the source is at the origin
    $\ve x = \ve 0 = (0,0,0)^T$ and emits a signal at time $t = 0$,
    then the $i$-th sensor will receive this at time
    $t_i = \lVert\ve a_i\rVert$. We have
    $(t_1,t_2,t_3,t_4,t_5) = (5,3,1,50/21,76/21)$.\ %
    We chose the~$\ve a_i$ in such a way that the affine relation
    $2 \ve a_1 + 2 \ve a_2 + 2 \ve a_3 - 3 \ve a_4 - 3 \ve a_5 = \ve 0$ is
    also satisfied by their lengths $\lVert\ve a_i\rVert$. So the
    matrix $A \in \RR^{5 \times 5}$ is not invertible.\ %
    \comment{
      Magma-code:
      
      ti:=[5,3,1,50/21,76/21];
      [ti[i]^2 - &+[ai[i][j]^2: j in [1..n]]: i in [1..n+2]];
      // all zero, so these really are the lengths
      ti:=Matrix(n+2,1,ti);
      A:=HorizontalJoin(-2*ti,Atilde);
      Determinant(A);
      // 0
    }%
    The computations according to~\eqref{1eqUV} yield
    \[
      \ve u = \frac{1}{55} \begin{pmatrix} 21 \\ 34 \\
        199 \end{pmatrix}, \ \ve v = \ve 0, \ \alpha =
      -\frac{152}{55}, \ \text{and} \ \beta = 0,
    \]
    so \eqref{1eqRank4} becomes
    \[
      \ve x = t \ve u \quad \text{and} \quad \frac{38173}{3025} t^2 +
      \frac{152}{55} t = 0,
    \]
    which is solved by
    \[
      t = 0, \ \ve x = \ve 0, \quad \text{and} \quad t' =
      \frac{-8360}{38173}, \ \ve x' =
      \frac{-152}{38173} \begin{pmatrix} 21 \\ 34 \\
        199 \end{pmatrix}.
    \]
    \comment{
      Magma-code:
      
      Btilde := (Transpose(Atilde)*Atilde)^(-1)*Transpose(Atilde);
      ua:=2*Btilde*ti;
      u:=Submatrix(ua,[1..n],[1]);
      alpha:=ua[n+1][1];
      u,alpha;
      // [ 21/55]
      // [  34/55]
      // [199/55]
      // -152/55
      u21:=&+[u[i,1]^2: i in [1..n]] - 1;
      u21;
      // 38173/3025
      ts:=alpha/u21;
      ts;
      // -8360/38173
      xs:=ts*u;
      -38173/152*xs;
      // [ 21]
      // [ 34]
      // [199]
    }%
    Now it can easily be verified directly that~\eqref{1eqBasic} holds
    with~$t$ and~$\ve x$ replaced by~$t'$ and~$\ve x'$. In other
    words, had the signal been emitted from the position~$\ve x'$ at
    time~$t'$ rather than from $\ve x = \ve 0$ at $t = 0$, it would
    have arrived at the exact same times~$t_i$ at the sensors. So even
    if the system~\eqref{1eqBasic} is overdetermined ($5$ equations
    for~$4$ unknowns) and the sensor positions are not chosen in an
    obviously clumsy way, it may still be impossible to uniquely
    determine the source position.%
    \comment{
      Magma-code:
      
      [ti[i,1] gt ts: i in [1..n+2]];
      // all true, so the t_i - t' are nonnegative
      [&+[(ai[i][j]-xs[j,1])^2: j in [1..n]] eq (ti[i,1] - ts)^2: i in
        [1..n+2]];
      // Again all true
    }
  \item \label{1exLateration2} A simpler example of the same type can
    be constructed in~$2$ dimensions. Take the sensor positions
    \[
      \ve a_1 = \begin{pmatrix} 9 \\ 12 \end{pmatrix}, \
      \ve a_2 = \begin{pmatrix} 9 \\ -12 \end{pmatrix}, \
      \ve a_3 = \begin{pmatrix} 10 \\ -24 \end{pmatrix}, \
      \ve a_4 = \begin{pmatrix} 10 \\ 24 \end{pmatrix}.
    \]
    The source is again located at $\ve x = \ve 0$ and emits its
    signal at time $t = 0$. It is then received at times
    $(t_1,t_2,t_3,t_4) = (15,15,26,26)$. The same computation as above
    (but with smaller numbers) shows that emission time and place
    \[
      t' = \frac{7}{5} \quad \text{and} \quad \ve x' =
      \frac{77}{5} \begin{pmatrix} 1 \\ 0 \end{pmatrix}
    \]
    would have led to the exact same arrival times~$t_i$.%
    \comment{
      Magma-code:

      n:=2;
      ai:=[[9,12],[9,-12],[10,-24],[10,24]];
      Atilde:=Matrix(n+2,n+1,[[Rationals()!2*ai[i][j]: j in [1..n]]
      cat [-1]: i in [1..n+2]]);
      [Determinant(Submatrix(Atilde,Exclude([1..n+2],i),[1..n+1])): i
      in [1..n+2]];
      // All nonzero
      ti:=[Rationals()!15,15,26,26];
      [ti[i]^2 - &+[ai[i][j]^2: j in [1..n]]: i in [1..n+2]];
      // all zero, so these really are the lengths
      ti:=Matrix(n+2,1,ti);
      A:=HorizontalJoin(-2*ti,Atilde);
      Determinant(A);
      // 0
      Btilde := (Transpose(Atilde)*Atilde)^(-1)*Transpose(Atilde);
      ua:=2*Btilde*ti;
      u:=Submatrix(ua,[1..n],[1]);
      alpha:=ua[n+1][1];
      u,alpha;
      // [11]
      // [0]
      // [199/55]
      // 168
      u21:=&+[u[i,1]^2: i in [1..n]] - 1;
      u21;
      // 120
      ts:=alpha/u21;
      ts;
      // 7/5
      xs:=ts*u;
      xs;
      // [77/5]
      // [   0]
      [ti[i,1] gt ts: i in [1..n+2]];
      // all true, so the t_i - t' are nonnegative
      [&+[(ai[i][j]-xs[j,1])^2: j in [1..n]] eq (ti[i,1] - ts)^2: i in
        [1..n+2]];
      // Again all true
    }%
  \item \label{1exLateration3} An example with a spurious solution
    (again in dimension~$2$) is given by
    \[
      \ve a_1 = \begin{pmatrix} 4 \\ 0 \end{pmatrix}, \
      \ve a_2 = \begin{pmatrix} -3 \\ 4 \end{pmatrix}, \
      \ve a_3 = \begin{pmatrix} -3 \\ -4 \end{pmatrix}, \
      \ve x = \begin{pmatrix} 0 \\ 0 \end{pmatrix}, \ \text{and} \ t = 0.
    \]
    Reception times are $(t_1,t_2,t_3) = (4,5,5)$. Here another
    solution of~\eqref{1eqAbsolute} is $t' = 28/3$ and
    $\ve x' = -\frac{4}{3} \left(\begin{smallmatrix} 1 \\ 0 \end
      {smallmatrix}\right)$. This is spurious, since $t_i - t' <
    0$. In fact, if the signal had been sent from~$\ve x'$ at
    time~$t'$, it would have arrived at times
    $(\frac{44}{3},\frac{41}{3},\frac{41}{3})$. The absolute
    differences between the arrival times are the same, but the
    sequence is reversed.%
    \comment{
      Magma-code:

      n:=2;
      ai:=[[4,0],[-3,4],[-3,-4]];
      Atilde:=Matrix(n+1,n+1,[[Rationals()!2*ai[i][j]: j in [1..n]]
      cat [-1]: i in [1..n+1]]);
      Determinant(Atilde);
      // nonzero
      ti:=[Rationals()!4,5,5];
      [ti[i]^2 - &+[ai[i][j]^2: j in [1..n]]: i in [1..n+1]];
      // all zero, so these really are the lengths
      ti:=Matrix(n+1,1,ti);
      Btilde := (Transpose(Atilde)*Atilde)^(-1)*Transpose(Atilde);
      ua:=2*Btilde*ti;
      u:=Submatrix(ua,[1..n],[1]);
      alpha:=ua[n+1][1];
      u,alpha;
      // [-1/7]
      // [0]
      // -64/7
      u21:=&+[u[i,1]^2: i in [1..n]] - 1;
      u21;
      // -48/49
      ts:=alpha/u21;
      ts;
      // 28/3
      [ti[i,1] gt ts: i in [1..n+1]];
      // all false, so the t_i - t' are negative -> spurious
      xs:=ts*u;
      xs;
      // [-4/3]
      // [   0]
      [&+[(ai[i][j]-xs[j,1])^2: j in [1..n]] eq (ti[i,1] - ts)^2: i in
        [1..n+1]];
      // All true
      // Arrival times if the signal had been sent at time ts:
      [ts + AbsoluteValue(ti[i,1]-ts): i in [1..n+1]];
      // [ 44/3, 41/3, 41/3 ]
    }%
  \item \label{1exLateration4} It can happen that the signal arrives at the same time
    $t_1 = \cdots = t_m$ at all sensors, so the source has the same
    distance from all of them. Intuition tells us that there can only
    be one such point, and again the mathematics bears this out. In
    fact, the equation $\tilde B \tilde A = I_{n+1}$
    before~\eqref{1eqTildeA} and the definition of $\tilde A$ imply
    that all rows from $\tilde{B}$, except for the last one, have
    coefficient sum~$0$. Therefore $\ve u = 0$, so indeed there is
    only one solution for~$\ve x$. Of the solutions for~$t$, one is
    spurious.
  \item \label{1exLateration5} Assume that~$\ve x$ is collinear with  two of the~$\ve a_i$, say
    $\ve a_1$ and~$\ve a_2$, but does not lie between them. Choosing
    the coordinate system suitably, we may assume~$\ve x = \ve 0$ and
    $t=0$. Then our assumption means $\ve a_1 = \lambda \ve a_2$ with
    $1 \ne \lambda > 0$. Now $\tilde B \tilde A = I_{n+1}$ implies
    that the last row of $\tilde B$ is
    $(\frac{-1}{\lambda-1},\frac{\lambda}{\lambda-1},0 \upto 0)$. We
    have
    $t_1 = \lVert\ve a_1\rVert = \lambda \lVert\ve a_2\rVert = \lambda
    t_2$, so $\alpha = 0$ by~\eqref{1eqUV}. Since $\ve v = \ve 0$ and
    $\beta = 0$, the quadratic equation becomes $t^2 = 0$.
  \item \label{1exLateration6} Here is an example where the
    coefficient of~$t^2$ becomes~$0$:
    \[
      \ve a_1 = \begin{pmatrix} 1 \\ 0 \end{pmatrix}, \
      \ve a_2 = \begin{pmatrix} -1 \\ 0 \end{pmatrix}, \
      \ve a_3 = \begin{pmatrix} 3 \\ 4 \end{pmatrix}, \
      \ve x = \begin{pmatrix} 0 \\ 0 \end{pmatrix}, \ \text{and} \ t = 0.
    \]
    The computation shows $\ve u = \left(\begin{smallmatrix} 0 \\
        1\end{smallmatrix}\right)$,  $\ve v = \left(\begin{smallmatrix} 0 \\
        0\end{smallmatrix}\right)$, $\alpha = -2$, and $\beta = 0$. So
    here the quadratic equation degenerates to $2 t = 0$.
    \comment{
      Magma-code:

      n:=2;
      ai:=[[1,0],[-1,0],[3,4]];
      Atilde:=Matrix(n+1,n+1,[[Rationals()!2*ai[i][j]: j in [1..n]]
      cat [-1]: i in [1..n+1]]);
      Determinant(Atilde);
      // nonzero
      ti:=[Rationals()!1,1,5];
      [ti[i]^2 - &+[ai[i][j]^2: j in [1..n]]: i in [1..n+1]];
      // all zero, so these really are the lengths
      ti:=Matrix(n+1,1,ti);
      Btilde := (Transpose(Atilde)*Atilde)^(-1)*Transpose(Atilde);
      ua:=2*Btilde*ti;
      u:=Submatrix(ua,[1..n],[1]);
      alpha:=ua[n+1][1];
      u,alpha;
      // [0]
      // [1]
      // -2
      u21:=&+[u[i,1]^2: i in [1..n]] - 1;
      u21;
      // 0
      // So the equation becomes 2 t = 0.
    }%
    This can be interpreted as follows: As the positions~$\ve a_i$
    approach the values given above, the alternative solution
    $\left(\begin{smallmatrix} t' \\ \ve x'\end{smallmatrix}\right)$
    of~\eqref{1eqRank4}, apart from the solution
    $\left(\begin{smallmatrix} t \\ \ve x\end{smallmatrix}\right) =
    \left(\begin{smallmatrix} 0 \\ \ve 0\end{smallmatrix}\right)$,
    tends to infinity. For example, taking
    $\ve a_3 = \left(\begin{smallmatrix} 3 \\
        3.99\end{smallmatrix}\right)$ and leaving~$\ve a_1$
    and~$\ve a_2$ unchanged leads to $t' \approx -1991$ and
    $\ve x' \approx \left(\begin{smallmatrix} 0 \\
        -1992\end{smallmatrix}\right)$.  \exend
    \comment{
      Magma-code:

      n:=2;
      ai:=[[1,0],[-1,0],[3,3.99]];
      Atilde:=Matrix(n+1,n+1,[[2*ai[i][j]: j in [1..n]]
      cat [-1]: i in [1..n+1]]);
      Determinant(Atilde);
      // nonzero
      ti:=[1,1,SquareRoot(ai[3][1]^2+ai[3][2]^2)];
      [ti[i]^2 - &+[ai[i][j]^2: j in [1..n]]: i in [1..n+1]];
      // all zero, so these really are the lengths
      ti:=Matrix(n+1,1,ti);
      Btilde := (Transpose(Atilde)*Atilde)^(-1)*Transpose(Atilde);
      ua:=2*Btilde*ti;
      u:=Submatrix(ua,[1..n],[1]);
      alpha:=ua[n+1][1];
      u,alpha;
      // [7.87599296662528609843919288905E-30]
      // [1.00050215683406730060441540339]
      // -2.00000000000000000000000000005
      u21:=&+[u[i,1]^2: i in [1..n]] - 1;
      u21;
      // 0.00100456582962060170330369035948
      ts:=alpha/u21;
      ts;
      // -1990.90984485839798474450427518
      [ti[i,1] gt ts: i in [1..n+1]];
      // all true, so the t_i - t' are nonnegative
      xs:=ts*u;
      xs;
      // [-1.56803919352897820431119774634E-26]
      // [-1991.90959384300549849073110292]
      [&+[(ai[i][j]-xs[j,1])^2: j in [1..n]] - (ti[i,1] - ts)^2: i in
        [1..n+1]];
      // Differences are negligible
    }%
  \end{enumerate}
  \renewcommand{\exend}{}
\end{ex}

Having seen from \cref{1exLateration}\ref{1exLateration1} that even
with~$5$ sensors such that no~$4$ of them are coplanar, it may happen that the matrix $A$,
defined in~\eqref{1eqA} is not invertible, we wonder how often this
happens. The following result says that under mild hypotheses, the
answer is ``very rarely'', i.e., almost certainly the
formula~\eqref{1eqB} can be applied for finding~$\ve x$ and~$t$.
Notice that $A$ is formed with times~$t_i$ given by~\eqref{1eqBasic},
so, because of the last column of $A$, its rank only depends on the
positions~$\ve a_i$ and~$\ve x$.

\begin{prop} \label{1pAInvertible}%
  Assume we have $m = 5$ sensors such that
  \begin{equation} \label{1eqCondition}%
    \det
    \begin{pmatrix}
      \varepsilon_1 \cdot \lVert\ve a_1\rVert & \ve a_1^T & 1 \\
      \vdots & \vdots & \vdots \\
      \varepsilon_5 \cdot \lVert\ve a_5\rVert & \ve a_5^T & 1
    \end{pmatrix} \ne 0 \quad \text{for all} \quad \varepsilon_1 \upto
    \varepsilon_5 = \pm 1.
  \end{equation}
  Then the set of all $\ve x \in \RR^3$ such that the matrix $A$ has
  rank $< 5$ is contained in a $2$-dimensional subvariety of $\RR^3$.
\end{prop}

\begin{proof}
  The function
  \[
    f(\ve x,\ve a_1 \upto \ve a_5) := \prod_{\varepsilon_1 \upto
    \varepsilon_5 = \pm 1} \det
    \begin{pmatrix}
      \varepsilon_1 \cdot \lVert\ve a_1 - \ve x\rVert & \ve a_1^T & 1 \\
      \vdots & \vdots & \vdots \\
      \varepsilon_5 \cdot \lVert\ve a_5 - \ve x\rVert & \ve a_5^T & 1
    \end{pmatrix}
  \]
  is a polynomial in the coefficients of~$\ve x$ and the~$\ve a_i$. If
  $A$ has rank $< 5$, then $f(\ve x,\ve a_1 \upto \ve a_5) = 0$. Our
  hypothesis means that $f(\ve 0,\ve a_1 \upto \ve a_5) \ne 0$, so the
  assertion follows.
\end{proof}

The hypothesis~\eqref{1eqCondition} in \cref{1pAInvertible} looks a
bit messy and lacks geometric content, but is readily verifiable. It would be desirable to have some more geometric conditions under which the assertion of \cref{1pAInvertible} holds. The following conjecture would be the best possible result, since its converse is clearly true.

\begin{conjecture} \label{1cAInvertible}%
  The assertion of \cref{1pAInvertible} holds under the milder
  hypothesis that the~$\ve a_i$ are not coplanar and pairwise
  distinct.
\end{conjecture}

We managed to prove the conjecture in the $2$-dimensional case by
considering the polynomial~$f$ used in the above proof as a polynomial
in the coordinates of~$\ve x$ as main variables, and forming the ideal
generated by the coefficients. A computation in the computer algebra
system MAGMA~[\citenumber{magma}] then shows that the equations that
express that the~$\ve a_i$ do {\em not} satisfy the hypothesis of the
conjecture all lie in the radical of this ideal. However, in the
$3$-dimensional case our computations ran into an impasse. The
$1$-dimensional case of the conjecture is false.

\begin{rem} \label{1rBancroft}%
  The first, and most cited, algebraic solution of the GPS-problem
  appears to have been given by \mycite{Bancroft:1985}
  in~\citeyear{Bancroft:1985}. Let us point out some differences
  between his approach and ours.
  \begin{itemize}
  \item Bancroft reaches a quadratic equation even if there are more
    than~$4$ sensors, and does not offer an explicit formula such
    as~\eqref{1eqB}.
  \item For reaching the equations~\eqref{1eqRank4} in the case that
    $A$ has rank $< 5$, we assume that $\tilde A$, defined
    in~\eqref{eqTildeAdef}, has rank~$4$ or, equivalently, that
    the~$\ve a_i$ are not coplanar. On the other hand, Bancroft
    assumes that the matrix
    \[
      \begin{pmatrix}
        \ve a_1^T & t_1 \\
        \vdots & \vdots \\
        \ve a_m^T & t_m
      \end{pmatrix} \in \RR^{m \times 4}
    \]
    has rank~$4$. But whether this is the case does not only depend on
    the positions~$\ve a_i$ and~$\ve x$, but also on~$t$, which in
    Bancroft's situation is the bias between the clocks. For example,
    if there are~$m = 4$ noncoplanar sensors, there is always a value
    for~$t$ such that the above matrix has determinant~$0$.
  \item There is no proof that Bancroft's quadratic equation does not
    degenerate. \remend
  \end{itemize}
  \renewcommand{\remend}{}
\end{rem}

\section{Relations and matching} \label{sMatching}

We consider the same situation as in the previous section: A source at
position $\ve x \in \RR^3$ emits a signal at time~$t$, which is
received by~$m$ sensors at positions $\ve a_1 \upto \ve a_m \in \RR^3$
and at times~$t_1 \upto t_m$. The unit of time is chosen such that the
signal propagates with speed~$1$, so 
$\lVert\ve a_i - \ve x\rVert = t_i - t$. The~$\ve a_i$ and~$t_i$ are
considered as known, and \cref{sMultilateration} was about how to
find~$\ve x$ and~$t$ from them. Now if there are~$5$ sensors or more,
then, according to the following result, there are algebraic relations
between the known quantities.

\begin{theorem} \label{2tRelations}%
  In the above situation, write $d_{i,j} := \lVert\ve a_i - \ve
  a_j\rVert$ and $t_{i,j} := t_i - t_j$. Then the matrix
  \[
    D = \bigl(t_{i,j}^2 - d_{i,j}^2\bigr)_{i,j = 1 \upto m} =
    \begin{pmatrix}
      0 & t_{1,2}^2 - d_{1,2}^2 & \cdots & t_{1,m}^2 - d_{1,m}^2 \\
      t_{2,1}^2 - d_{2,1}^2 & 0 & \cdots & t_{2,m}^2 - d_{2,m}^2 \\
      \vdots & \vdots & \ddots & \vdots \\
      t_{m,1}^2 - d_{m,1}^2 & t_{m,2}^2 - d_{m,2}^2 & \cdots & 0
    \end{pmatrix}
    \in \RR^{m \times m}
  \]
  has rank $\le 4$. So if $m \ge 5$ we have the relation
  $\det(D) = 0$, and if $m > 5$ all $5 \times 5$-minors are zero.
\end{theorem}

We have formulated \cref{2tRelations} in the $3$-dimensional case for the sake of simplicity. But it actually holds in any dimension~$n$, saying that $\rank(D) \le n+1$.

For the proof we will use the Cayley-Menger matrix. The following
result about its rank is well known (see \mycite{Cayley:1841}) in the
case of Euclidean spaces, but we need a more general version given by
the following proposition. We will not need the exact value of the
rank, but include it for the sake of completeness.

\begin{prop} \label{2pCayleyMenger}%
  Let $\ve v_0 \upto \ve v_m \in V$ be vectors in a Euclidean space
  or, more generally, in a vector space over a field of characteristic
  $\ne 2$ equipped with a quadratic form~$q$. Set
  $\delta_{i,j} := q(\ve v_i - \ve v_j)$, which in the special case of
  a Euclidean space is the squared distance between~$\ve v_i$
  and~$\ve v_j$. Then the {\bf Cayley-Menger matrix}
  \[
    C := \begin{pmatrix}
      0 & 1 & 1 & 1 & \cdots & 1 \\
      1 & 0 & \delta_{0,1} & \delta_{0,2} & \cdots & \delta_{0,m} \\
      1 & \delta_{1,0} & 0 & \delta_{1,2} & \cdots & \delta_{1,m}  \\
      1 & \delta_{2,0} & \delta_{2,1} & 0 & \cdots & \delta_{2,m} \\
      \vdots & \vdots & \vdots & \vdots & \ddots & \vdots \\
      1 & \delta_{m,0} & \delta_{m,1} & \delta_{m,2} & \cdots & 0
    \end{pmatrix} \in K^{(m+2) \times (m+2)}
  \]
  has rank $\le \dim(V) + 2$. More precisely, if~$r$ is the rank
  of~$q$ restricted to the subspace $U \subseteq V$ generated by
  $\ve v_1 - \ve v_0 \upto \ve v_m - \ve v_0$, then
  $\rank(C) = r + 2$. Notice that $r \le \dim(U)$, with equality in
  the special case of a Euclidean space. Also notice that $\dim(U)$ is
  equal to the dimension of the {\em affine} subspace generated by
  $\ve v_0 \upto \ve v_m$.
\end{prop}

\begin{proof}
  We only need to show $\rank(C) = r + 2$ in the more general case of
  a quadratic space over a field $K$. Replacing each~$\ve v_i$ by
  $\ve v_i - \ve v_0$ does not change $C$ or $U$, so by doing this we
  may assume $\ve v_0 = \ve 0$. Then $U$ is generated (as a vector
  space) by $\ve v_1 \upto \ve v_m$, and we may replace $V$ by
  $U$. Now $V$ is generated by the $\ve v_i$ and in particular
  finite-dimensional. By choosing a basis of $V$ we may then replace
  $V$ by $K^n$. Then~$q$ is given by $q(\ve v) = \ve v^T A \ve v$ for
  $\ve v \in V = K^n$, with $A \in K^{n \times n}$ a symmetric matrix
  of rank~$r$. $A$ also defines the bilinear form
  $\langle\cdot,\cdot\rangle$ belonging to~$q$. After these
  reductions, the main part of the proof rests on the following matrix
  computation.

  With
  \[
    E:= \left(
      \begin{array}{c|c}
        \begin{array}{cc}
          0 & 1 \\
          1 & 0
        \end{array} & 0 \\
        \hline
        \vphantom{3^{3^{3^3}}} 0 & -2 A
      \end{array}\right) \in K^{(n+2) \times (n+2)}, \ %
    F := \left(
      \begin{array}{c|c}
        I_2 & \begin{array}{ccc}
                \delta_{0,1} & \cdots & \delta_{0,m} \\
                1 & \cdots & 1
              \end{array} \\
        \hline
        \vphantom{3^{3^{3^3}}} 0 & 
                                   \begin{array}{ccc}
                                     \ve v_1 & \cdots & \ve v_m
                                   \end{array}
      \end{array}\right) \in K^{(n+2) \times (m+2)}
  \]
  we have
  \begin{multline*}
    F^T E F = \left(
      \begin{array}{c|c}
        I_2 & \vphantom{3_{3_{3_3}}} 0 \\
        \hline
        \begin{array}{cc}
          \delta_{0,1} & \vphantom{3^{3^{3^3}}} 1 \\
          \vdots & \vdots \\
          \delta_{0,m} & 1
        \end{array} &
                      \begin{array}{c}
                        \ve v_1^T \\
                        \vdots \\
                        \ve v_m^T
                      \end{array}
      \end{array}\right) %
    \left(
      \begin{array}{c|c}
        \begin{array}{cc}
          0 & 1 \\
          1 & 0
        \end{array} & \begin{array}{ccc}
                        1 & \cdots & \hphantom{A} 1 \\
                        \hphantom{a} \delta_{0,1} & \cdots &
                                                             \hphantom{-2}
                                                             \delta_{0,m}
                      \end{array} \\
        \hline
        \vphantom{3^{3^{3^3}}} 0 &  \begin{array}{ccc}
                                      -2 A \ve v_1 & \cdots & -2 A \ve v_m
                                    \end{array}
      \end{array}\right) \\ %
    = \left(
      \begin{array}{c|c}
        \begin{array}{cc}
          0 & \hphantom{2} 1 \hphantom{2} \\
          1 & \hphantom{2} 0 \hphantom{2}
        \end{array} & \begin{array}{ccc}
                        1 & \cdots & 1 \\
                        \delta_{0,1} & \cdots & \delta_{0,m}
                      \end{array} \\
        \hline
        \begin{array}{cc}
          \vphantom{3^{3^{3^3}}} 1 & \delta_{1,0} \\
          \vdots & \vdots \\
          1 & \delta_{m,0}
        \end{array} & B
      \end{array}\right),
  \end{multline*}
  where the $(i,j)$-th entry of the matrix $B$ is
  \[
    \delta_{0,i} + \delta_{0,j} - 2 \ve v_i^T A \ve v_j = \langle\ve
    v_i,\ve v_i\rangle + \langle\ve v_j,\ve v_j\rangle - 2 \langle\ve
    v_i,\ve v_j\rangle = \langle\ve v_i - \ve v_j,\ve v_i - \ve
    v_j\rangle = \delta_{i,j}.
  \]
  So $F^T E F = C$. Since the~$\ve v_i$ span $V = K^n$, $F$ has
  rank~$n + 2$, so the linear map $K^{m+2} \to K^{n+2}$ given by $F$
  is surjective. Likewise, the map given by $F^T$ is injective, and
  the image of the map given by $E$ has dimension equal to
  $\rank(A) + 2 = r + 2$. It follows that the map given by $C$ has an
  image of dimension~$r + 2$, which is our claim.
\end{proof}

\begin{proof}[Proof of \cref{2tRelations}]
  We apply \cref{2pCayleyMenger} with $V = \RR^4$ being Minkowski
  space, so for
  $\ve v = \left(\begin{smallmatrix} t \\ \ve
      u\end{smallmatrix}\right) \in V$ with $t \in \RR$,
  $\ve u \in \RR^3$ the quadratic form is
  $q(\ve v) = t^2 - \lVert\ve u\rVert^2$, with $\lVert\cdot\rVert$ the
  usual Euclidean norm. For $i = 1 \upto m$, set
  $\ve v_i := \left(\begin{smallmatrix} t_i \\ \ve
      a_i\end{smallmatrix}\right)$, and also set
  $\ve v_0 := \left(\begin{smallmatrix} t \\ \ve
      x\end{smallmatrix}\right)$ (the time and place of
  emission). Then the equation $\lVert\ve a_i - \ve x\rVert = t_i - t$
  implies $\delta_{0,i} = q(\ve v_0 - \ve v_i) = 0$ for all~$i$. So
  the Cayley-Menger matrix becomes
  \[
    C = \left(
      \begin{array}{c|c}
        \begin{array}{cc}
          0 & 1  \\
          1 & 0
        \end{array} & \begin{array}{ccc}
                        1 & \cdots & 1 \\
                        0 & \cdots & 0
                      \end{array} \\
        \hline
        \begin{array}{cc}
          \vphantom{3^{3^3}} 1 & 0 \\
          \vdots & \vdots \\
          1 & 0
        \end{array} & D
      \end{array}\right),
  \]
  with $D$ as defined in the theorem. Since $\rank(C) \le 6$ by
  \cref{2pCayleyMenger}, we get $\rank(D) \le 4$.
\end{proof}

By \cref{2tRelations} there is a relation between the reception times
of a signal coming from a single emission event. If there are multiple
emission events (from different source locations and/or at different
times), this relation can be used to match those reception times that
come from the same event. Matching reception times can then be used to
determine the time and place of emission, making use of the methods
from \cref{sMultilateration}. \cref{2aDetect} makes this idea precise.

\begin{alg}
  \caption{Detect source positions and emission times of multiple
    emission events} \label{2aDetect} \mbox{}%
  \begin{description}
  \item[Input] For $i = 1 \upto 5$, a set $\mathcal T_i$ containing
    the points in time when the $i$th sensor received a signal. The
    units of time and distance should be chosen such that signals
    travel with speed~$1$. The positions
    $\ve a_1 \upto \ve a_5 \in \RR^3$ of the sensors need to be known.
  \end{description}
  \begin{algorithmic}[1]
    \STATE \label{2aDetect1} Let
    $d_{i,j} := \lVert\ve a_i - \ve a_j\rVert$ be the distances
    between the sensors. Set $\mathcal E := \emptyset$. The detected
    emission events will be collected as pairs $(t,\ve x)$ in the set
    $\mathcal E$.
    \FOR{$(t_1,t_2,t_3,t_4,t_5) \in \mathcal T_1 \times \mathcal T_2
      \times \mathcal T_3 \times \mathcal T_4 \times \mathcal
      T_5$} \label{1aDetect2}%
    \STATE \label{2aDetect3} Set up the matrix
    $D = \bigl((t_i - t_j)^2 - d_{i,j}^2\bigr)_{i,j = 1 \upto 5}$.%
    \IF{$\det(D) = 0$} \label{1aDetect4}%
    \STATE \label{2aDetect5} Use~\cref{1tLateration}, with the current
    $t_1 \upto t_5$ and $\ve a_1 \upto \ve a_5$ as input, to compute
    the emission time~$t$ and the source position~$\ve x$. Include
    $(t,\ve x)$ in the set $\mathcal E$. In the unlikely event
    that~\cref{1tLateration} yields two solutions and neither can be
    discarded (as spurious or from other context), include both.%
    \ENDIF%
    \ENDFOR%
    \STATE \label{2aDetect8} {\bf Output} the set $\mathcal E$ of
    detected emission events.%
  \end{algorithmic}
\end{alg}

\begin{rem} \label{3rSpurious}%
In fact, \cref{2aDetect} does more than just matching reception times of signals. It also discards spurious signals registered by sensors. By this we mean erroneous registrations of signals, or registrations of signals that originate very near to a sensor and are irrelevant since they cannot be perceived by other sensors. Indeed, it is almost impossible for such a spurious signal to satisfy the relation $\det(D) = 0$ together with other "legitimate" reception times.
\end{rem}

In a later paper we will study the behavior (and modifications) of the
algorithm in situations where the input data is
inexact because of measurement errors.

\section{Wall detection by echoes} \label{sEchoes}

\cref{2tRelations} guarantees that the relation $\det(D) = 0$ always holds if the
signals received at times $t_1 \upto t_5$ come from the same emission
event. So all events for which a signal is received by every sensor
will be detected. But it is possible that the determinant becomes~$0$
even if the signals do not come from the same event. For instance, if
$t_1 \upto t_4$ do come from the same event, a different source may
send a signal at such an unlucky time that it is received by the fifth
sensor at a time~$t_5$ that happens to make the equation $\det(D) = 0$
come true. This can happen no matter where sources and sensors are
positioned. In fact, with the signal reception times as the only available information, any algorithm would be tricked into making a false match if some signal arrives at an unlucky time.

We will now restrict our attention to the case in which all emission
events share the same emission time (which is still unknown to the
sensors). This happens when the emission events are in fact echoes of
a single sound emission bouncing off from various walls (i.e., flat
surfaces). In fact, in the ray acoustics approximation the received
echoes virtually come from the so-called {\em mirror points}, i.e.,
the points obtained by reflecting the original source position at the
walls. If an echo from a wall is received by all five sensors (or
microphones in the acoustic case), then \cref{2aDetect} computes the
mirror point, and it is easy to find the wall position from this. With
the restriction to simultaneous (virtual) emission events, it becomes
less likely that \cref{2aDetect} produces a mismatch ($\det(D) = 0$ even
though the signals come from different mirror points) and so
erroneously detects a wall which is not really there (often called a
{\em ghost wall}).

So we can be hopeful that, in contrast to the situation with different
emission times, the choice of the sensor positions may preclude ghost
walls. To give this a name, we say that the sensors are in a {\em good
  position} if \cref{2aDetect} produces no ghost walls. Whether this
is true clearly also depends on the coordinates of the mirror points,
but not on the time of the sound emission, since only differences of
reception times go into the matrix $D$. Even more hopeful, we say that
{\em almost all} positions are good if the bad positions are contained
in a lower-dimensional subvariety of the configuration space
$(\RR^3)^5$ of all possible microphone positions. Intuitively ``almost
all'' can be thought of as ``with probability one.''

\begin{theorem} \label{3tGood}%
  Consider a given room, by which we understand an arrangement of
  walls, which may include ceilings, floors, and sloping walls. Assume
  there is a loudspeaker at a given position in the room. Now five
  microphones are positioned in the room. Then almost all loudspeaker
  positions are good, meaning that from a single sound emitted by the
  loudspeaker, \cref{2aDetect} detects all walls from which an echo is
  received by every microphone, but it detects no ghost walls.
\end{theorem}

\begin{proof}
  We are given a finite set $\mathcal S \subset \RR^3$ of mirror
  points, obtained by reflecting the given position of the loudspeaker
  at the various walls. From now on we can forget about the walls and
  the loudspeaker, since the signals will be received by the
  microphones as if they were all simultaneously emitted from the
  mirror points. The microphone positions can be represented as the
  columns of a matrix
  $M = (\ve a_1 \upto \ve a_5) \in \RR^{3 \times 5}$. We can thus
  speak of good or bad matrices $M$. We will also say that $M$ is {\bf
    very good} if the following is true: For any five points
  $\ve s_1 \upto \ve s_5 \in \mathcal S$, the relation
  \begin{equation} \label{3eqVery}%
    \det\Bigl(\bigl(\lVert\ve s_i - \ve a_i\rVert - \lVert\ve s_j -
    \ve a_j\rVert\bigl)^2 - \lVert\ve a_i - \ve a_j\rVert^2\Bigr)_{i,j
      = 1 \upto 5} = 0
  \end{equation}
  only holds if $\ve s_1 = \ve s_2 = \ve s_3 = \ve s_4 = \ve s_5$. A
  very good position is a good one, since for a quintuple
  $(t_1 \upto t_5) \in \mathcal T_1 \times \cdots \times \mathcal T_5$
  considered in \cref{2aDetect} there are points
  $\ve s_i \in \mathcal S$ such that
  $t_i - t = \lVert\ve s_i - \ve a_i\rVert$ (with~$t$ the time of
  simultaneous emission). So the determinant considered in the
  algorithm is just the one in~\eqref{3eqVery}. If this is zero,
  the hypothesis of ``very good'' implies that the $\ve s_i$ are all
  the same point $\ve s \in \mathcal S$. So indeed if the algorithm
  includes a point into the set $\mathcal E$, this will be an actual
  mirror point, meaning that $M$ is good.

  The main difficulty of the proof arises from the fact that the
  determinant in~\eqref{3eqVery} is not a polynomial in the
  coordinates of the~$\ve a_i$ and the~$\ve s_i$, because the norms
  involve square roots. To circumvent this problem, we form the
  product
  \[
    f(\ve a_1 \upto \ve a_5,\ve s_1 \upto \ve s_5) := \!\!\!\!
    \prod_{\substack{\varepsilon_1 \upto \varepsilon_4 = \pm 1 \\
        \epsilon_5 = 1}} \!\!\!\! \det\Bigl(\bigl(\varepsilon_i
    \lVert\ve s_i - \ve a_i\rVert - \varepsilon_j \lVert\ve s_j - \ve
    a_j\rVert\bigl)^2 - \lVert\ve a_i - \ve a_j\rVert^2\Bigr)_{i,j = 1
      \upto 5},
  \]
  which is easily seen to be a polynomial in the coordinates of its
  arguments. Let us say that $M$ is {\bf excellent} if for any
  $\ve s_1 \upto \ve s_5 \in \mathcal S$ the relation
  $f(\ve a_1 \upto \ve a_5,\ve s_1 \upto \ve s_5) = 0$ implies that
  the~$\ve s_i$ are all equal. So ``excellent'' implies ``very good''
  and ``good''. Now the set
  \[
    \mathcal U_{\ve s_1 \upto \ve s_5} := \bigl\{M = (\ve a_1 \upto
    \ve a_5) \in \RR^{3 \times 5} \mid f(\ve a_1 \upto \ve a_5,\ve s_1
    \upto \ve s_5) \ne 0\bigr\}
  \]
  is Zariski open in $\RR^{3 \times 5}$, and
  \[
    \mathcal U := \bigcap_{\substack{\ve s_1 \upto \ve s_5 \in S\
        \text{such that} \\ \text{not all}\ \ve s_i\ \text{are
          equal}}} \mathcal U_{\ve s_1 \upto \ve s_5}
  \]
  is the set of excellent matrices. So if we can show that
  $\mathcal U_{\ve s_1 \upto \ve s_5} \ne \emptyset$ for all
  $\ve s_1 \upto \ve s_5 \in \mathcal S$ that are not all equal, by
  the Zariski openness the theorem follows. Since we have no control
  over the given set $\mathcal S$, we need to show the nonemptiness
  for any five points $\ve s_i \in \RR^3$ that are not all equal, and
  this also suffices. Equivalently, we need to prove the following.

  \noindent {\bf Claim.} If $\ve s_1 \upto \ve s_5 \in \RR^3$ are
  points such that
  \[
    f(\ve a_1 \upto \ve a_5,\ve s_1 \upto \ve s_5) = 0 \quad \text{for
      all} \quad M = (\ve a_1 \upto \ve a_5) \in \RR^{3 \times 5},
  \]
  then $\ve s_1 = \ve s_2 = \cdots = \ve s_5$.

  Having reduced the proof to the claim, we can forget about the
  situation of the theorem. In principle, the claim could be proved by
  explicitly forming the polynomial~$f$, regarding the variables for
  the $\ve a_i$ as main variables, extracting the coefficients (which
  are polynomials in the $\ve s_i$-variables), and showing that the
  ideal generated by the coefficients defines the variety given by
  $\ve s_1 = \ve s_2 = \cdots = \ve s_5$.  Unfortunately, $f$ has 15
  variables and is homogeneous of degree~160. The first step towards
  making the computation feasible is choosing suitable Cartesian
  coordinates as follows. The vector $\ve s_1$ can be taken as the
  origin of the coordinate system. This turns the matrix $S := (\ve s_1
  \upto \ve s_5) \in \RR^{3 \times 5}$ into $S = (\ve 0,\ve s_2 \upto
  \ve s_5)$. We can now apply QR-decomposition, i.e., write
  \begin{equation} \label{3eqS}%
    S = Q \cdot
    \begin{pmatrix}
      0 & b_1 & b_2 & b_3 & b_4 \\
      0 & 0 & b_5 & b_6 & b_7 \\
      0 & 0 & 0 & b_8 & b_9
    \end{pmatrix} =: Q B
  \end{equation}
  with $Q \in \SO(3)$. So using the columns of $Q$ as a new basis of
  $\RR^3$, $S$ becomes the above upper triangular matrix $B$:
  $(\ve s_1 \upto \ve s_5) = S = B$.

  Form the matrices
  \[
    X =
    \begin{pmatrix}
      x_{1,1} & \cdots & x_{1,5} \\
      \vdots & & \vdots \\
      x_{3,1} & \cdots & x_{3,5}
    \end{pmatrix} \quad \text{and} \quad
    Y = (Y_{i,j}) =
    \begin{pmatrix}
      0 & y_1 & y_2 & y_3 & y_4 \\
      0 & 0 & y_5 & y_6 & y_7 \\
      0 & 0 & 0 & y_8 & y_9
    \end{pmatrix}
  \]
  with $x_{i,j}$ and~$y_i$ indeterminates. With additional
  indeterminantes $z_1 \upto z_5$, form the ideal
  \[
    J := \Bigl(z_j^2 - \sum_{i=1}^3 (Y_{i,j} -x_{i,j})^2 \bigl| j = 1
    \upto 5\Bigr),
  \]
  with the idea that the $z_j$ stand for $\lVert\ve s_i - \ve
  a_i\rVert$. Modulo $J$, the product
  \begin{equation} \label{3eqF}%
    \prod_{\substack{\varepsilon_1 \upto \varepsilon_4 = \pm 1 \\
        \epsilon_5 = 1}} \!\!\!\! \det\left((\varepsilon_i z_i -
      \varepsilon_j z_j)^2 - \sum_{k=1}^3 (x_{k,i} -
      x_{k,j})^2\right)_{i,j = 1 \upto 5}
  \end {equation}
  reduces to a polynomial $F(x_{1,1} \upto x_{3,5},y_1 \upto y_9)$
  which does not involve the~$z_j$. If we specialize the variables in
  the matrix $X$ to the entries in a matrix
  $M = (\ve a_1 \upto \ve a_5) = (a_{i,j}) \in \RR^{3 \times 5}$ and
  the variables in $Y$ to the matrix $B$ in~\eqref{3eqS}, we obtain
  \[
    F(a_{1,1} \upto a_{3,5},b_1 \upto b_9) = f(\ve a_1 \upto \ve
    a_5,\ve s_1 \upto \ve s_5).
  \]
  So to prove the above claim and thus the theorem, we need to form
  $F(x_{1,1} \upto x_{3,5},y_1 \upto y_9)$, regard the $x_{i,j}$ as
  main variables and consider the ideal
  $L \subseteq \RR[y_1 \upto y_9]$ generated by the coefficients. Then
  we need to show that $L$ has $y_1 = \cdots = y_9 = 0$ as the only
  real solution, which says that all~$\ve s_i$ are zero and hence
  equal.

  Alas, even after reducing
  the number of variables by our choice of coordinates, computing the
  polynomial $F$ is still utterly impossible. What we did instead was
  setting almost all of the variables $x_{i,j}$ to zero and compute
  the product~\eqref{3eqF} with these specializations, always reducing
  modulo $J$. This turns out to be possible in many cases, and
  extracting coefficients gave us {\em some} generators of $L$. Doing
  this for many choices of specialized variables provides ever more
  generators of $L$. Each time, we also reduced modulo the generators
  of $L$ already known, which is permissible and accelerates the
  computation. When we found a sum of squares of variables in the
  ideal, we substituted this by the variables themselves since only
  real solutions need to be considered.

  With this technique, using random specializations of variables, we
  eventually arrived at an ideal whose only solution is the origin
  $y_i = 0$. We recorded exactly which sequence of specializations led
  to this result and from this produced a deterministic, reproducible
  procedure for verifying the claim. All computations were done in
  MAGMA~[\citenumber{magma}].
\end{proof}

\begin{rem*}
\cref{3tGood} gives the theoretical justification for a procedure that detects walls by solving the pseudo-range multilateration problem for each wall. The following alternative method comes to mind. Since the sound traveling directly from the loudspeaker to the microphones always arrives first, before any echoes, the very first signals registered by the microphones must be the ones coming directly from the loudspeaker. Therefore these can be used, without any matching process, to determine the emission time~$t$ by multilateration. Since the echoes virtually come from the mirror points and are emitted at the same time~$t$ which is now known, the methods from \mycite{DPWLV1} or \mycite{Boutin:Kemper:2019} can then be used for the wall detection.

While this alternative would presumably work in many cases, it has some drawbacks. For one thing, obstacles may in some cases prevent one or more microphones from hearing the direct signal from the loudspeaker. Secondly, and perhaps more importantly, a spurious signal may be registered by one or more of the microphones before the true signal from the loudspeaker arrives. In this case the alternative method would mistake this spurious signal as the direct signal. Therefore the multilateration would produce a drastically wrong time of emission, and all subsequent wall detections, based on this erroneous time, would become false.
\end{rem*}

\bibliographystyle{mybibstyle} \bibliography{bib}

@preamble{ "\newcommand{\SortNoop}[1]{}" }

@article{sha2021reliability,
  title={Reliability of Seismic Signal Analysis for Earthquake Epicenter Location Estimation Using 1 Hz GPS Kinematic Solution},
  author={Sha'ameri, AZ and Aris, WA Wan and Sadiah, S and Musa, TA},
  journal={Measurement},
  volume={182},
  pages={109669},
  year={2021},
  publisher={Elsevier}
}

@inproceedings{zannini2010improved,
  title={Improved TDOA disambiguation techniques for sound source localization in reverberant environments},
  author={Zannini, Cecilia Maria and Cirillo, Albenzio and Parisi, Raffaele and Uncini, Aurelio},
  booktitle={Proceedings of 2010 IEEE International Symposium on Circuits and Systems},
  address={},
  pages={2666--2669},
  year={2010},
  organization={IEEE}
}

@inproceedings{el2017time,
  title={Time of arrival disambiguation using the linear Radon transform},
  author={El Baba, Youssef and Walther, Andreas and Habets, Emanu{\"e}l AP},
  booktitle={2017 IEEE International Conference on Acoustics, Speech and Signal Processing (ICASSP)},
    address={},
  pages={106--110},
  year={2017},
  organization={IEEE}
}

@inproceedings{kreissig2013fast,
  title={Fast and reliable TDOA assignment in multi-source reverberant environments},
  author={Krei{\ss}ig, Martin and Yang, Bin},
  booktitle={2013 IEEE International Conference on Acoustics, Speech and Signal Processing},
    address={},
  pages={355--359},
  year={2013},
  organization={IEEE}
}

@article{lundberg2001alternative,
  title={Alternative algorithms for the GPS static positioning solution},
  author={Lundberg, John B},
  journal={Applied mathematics and computation},
  volume={119},
  number={1},
  pages={21--34},
  year={2001},
  publisher={Elsevier}
}

@inproceedings{li2010design,
  title={Design and analysis of a new gps algorithm},
  author={Li, Wei and Li, Deng and Yang, Shuhui and Xu, Zhiwei and Zhao, Wei},
  booktitle={2010 IEEE 30th International Conference on Distributed Computing Systems},
    address={},
  pages={40--51},
  year={2010},
  organization={IEEE}
}

@article{magma,
  author =	 {Wieb Bosma and John J. Cannon and Catherine
                  Playoust},
  title =	 {The {M}agma Algebra System~{I}: The User Language},
  journal =	 {J.~Symb. Comput.},
  volume =	 24,
  year =	 1997,
  pages =        {235--265}
}

@Article{DPWLV1,
  author = 	 {Ivan Dokmani{\'{c}} and Reza Parhizkar and Andreas
                  Walther and Yue M. Lu and Martin Vetterli},
  title = 	 {Acoustic echoes reveal room shape},
  journal = 	 {Proceedings of the National Academy of Sciences},
  year = 	 2013,
  volume = 	 110}

@Article{Boutin:Kemper:2019,
  author = 	 {Mireille Boutin and Gregor Kemper},
  title = 	 {A Drone Can Hear the Shape of a Room},
  journal = 	 {SIAM J. Appl. Algebra Geometry},
  year = 	 2020,
  volume = 	 4,
  pages = 	 {123--140}
}

@Article{Cayley:1841,
  author = 	 {Arthur Cayley},
  title = 	 {On a theorem in the geometry of position},
  journal = 	 {Cambridge Mathematical Journal},
  year = 	 1841,
  volume = 	 {II},
  pages = 	 {267-–271}
}

@Article{Bancroft:1985,
  author = 	 {Stephen Bancroft},
  title = 	 {An algebraic solution of the {GPS} equations},
  journal = 	 {IEEE Transactions on Aerospace and Electronic Systems},
  year = 	 1985,
  volume = 	 7,
  pages = 	 {56--59}
}

@Article{Krause:1987,
  author = 	 {Lloyd O. Krause},
  title = 	 {A Direct Solution to {GPS}-Type Navigation Equations
},
  journal = 	 {IEEE Transactions on Aerospace and Electronic Systems},
  year = 	 1987,
  volume = 	 23,
  pages = 	 {225--232}
}

@Article{Chaffee:Abel:1994,
  author = 	 {James Chaffee and Jonathan S. Abel},
  title = 	 {On the Exact Solutions of Pseudorange Equations},
  journal = 	 {IEEE Transactions on Aerospace and Electronic Systems},
  year = 	 1994,
  volume = 	 30,
  pages = 	 {1021--1030}
}

@Article{Beck:Pan:2012,
  author = 	 {Amir Beck and Dror Pan},
  title = 	 {On the solution of the {GPS} localization and circle
                  fitting problems},
  journal = 	 {SIAM J. Optim.},
  year = 	 2021,
  volume = 	 22,
  pages = 	 {108--134}
}

@inproceedings{jager2016room,
  title={Room geometry estimation from acoustic echoes using graph-based echo labeling},
  author={Jager, Ingmar and Heusdens, Richard and Gaubitch, Nikolay D},
  booktitle={2016 IEEE International Conference on Acoustics, Speech and Signal Processing (ICASSP)},
  address={},
  pages={1--5},
  year={2016},
  organization={IEEE}
}

\end{document}